\documentclass{article}

\usepackage[a4paper,scale=0.75]{geometry}
\usepackage{amsmath,amssymb,amsthm}
\usepackage{graphicx,color,ifthen}
\usepackage{bm}

\def\CC{\mathbb{C}}
\def\NN{\mathbb{N}}
\def\RR{\mathbb{R}}
\def\ZZ{\mathbb{Z}}

\providecommand{\abs}[1]{{\lvert#1\rvert}}
\providecommand{\floor}[1]{{\lfloor#1\rfloor}}
\providecommand{\bigfloor}[1]{{\bigl\lfloor#1\bigr\rfloor}}

\providecommand{\ceil}[1]{{\lceil#1\rceil}}
\providecommand{\bigceil}[1]{{\bigl\lceil#1\bigr\rceil}}
\providecommand{\ideal}[1]{{\langle#1\rangle}}
\providecommand{\bigideal}[1]{{\bigl\langle#1\bigr\rangle}}

\def\dG{m}     
\def\dR{m'}    

\newcommand\balpha{{\bm{\alpha}}}
\newcommand\bbeta{{\bm{\beta}}}
\newcommand\bgamma{{\bm{\gamma}}}
\newcommand\bdelta{{\bm{\delta}}}
\newcommand\bsigma{{\bm{\sigma}}}
\newcommand\bpi{{\bm{\pi}}}
\newcommand\bzero{{\bm{0}}}

\newcommand\bk{\bm{k}}
\newcommand\bz{\bm{z}}

\def\cI{\mathcal{I}}
\def\cJ{\mathcal{J}}

\newtheorem{theorem}{Theorem}[section]
\newtheorem{proposition}[theorem]{Proposition}

\newtheorem{definition}[theorem]{Definition}
\newtheorem{example}[theorem]{Example}
\newtheorem{remark}[theorem]{Remark}

\newcommand{\rev}[1]{{#1}}

\newcommand{\support}[3][]{
\setlength{\unitlength}{2.5ex}
\text{\small\parbox[c]{4\unitlength}{\centering
\ifthenelse{\equal{#1}{}}{
\begin{picture}(3,3)(0,0)
  \put(0,0){\includegraphics[width=3\unitlength]{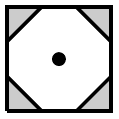}}
  \put(-0.15,2.5){\makebox(0,0)[cr]{$#3$}}
  \put(3.15,1.5){\makebox(0,0)[cl]{$#2$}}
\end{picture}}{
\begin{picture}(3,4)(0,-1)
  \put(0,0){\includegraphics[width=3\unitlength]{support2}}
  \put(-0.15,2.5){\makebox(0,0)[cr]{$#3$}}
  \put(3.15,1.5){\makebox(0,0)[cl]{$#2$}}
  \put(1.5,-0.2){\makebox(0,0)[tc]{$#1$}}
\end{picture}}}}}

\begin{document}

\title{Symmetric four-directional bivariate pseudo-splines}
\author{Costanza Conti \and Chongyang Deng \and Kai Hormann}
\date{}
\maketitle


\begin{abstract}
Univariate pseudo-splines are a generalization of uniform B-splines and interpolatory $2n$-point subdivision schemes. Each pseudo-spline is characterized as the subdivision scheme with least possible support among all schemes with specific degrees of polynomial generation and reproduction. In this paper we consider the problem of constructing the symbols of the bivariate counterpart and provide a formula for the symbols of a family of symmetric four-directional bivariate pseudo-splines. All methods employed  are of purely algebraic nature.
\end{abstract}


\section{Introduction}\label{sec:introduction}
Subdivision schemes are efficient iterative tools for generating curves, surfaces, wavelets, and frame constructions used in many fields ranging from computer aided geometric design to signal and image processing. Important properties of subdivision schemes such as convergence, regularity, polynomial generation, and approximation order have been studied intensively in the last twenty years, often by the help of the so-called \emph{subdivision symbol}, a Laurent polynomial associated with the subdivision scheme (see the surveys~\cite{Cavaretta:1991:SS} and~\cite{Dyn:2002:SSI} and the references therein). Indeed, the algebraic properties of the symbol translate directly into analytical properties of the corresponding subdivision scheme and its limit~\cite{Dyn:2008:PRB,Charina:2011:SMS,Conti:2011:PRF,Conti:2013:ACA}.

In this paper we are particularly interested in the properties of \emph{polynomial generation} and \emph{polynomial reproduction}. The former is the capability of a subdivision scheme to generate a certain space of polynomials in the limit, and the latter is the capability to produce exactly the same polynomial from which the initial data is sampled. Both properties are very important, because a high generation degree indicates a potentially high regularity of the scheme, while a high reproduction degree implies a high approximation order~\cite{Jia:2002:APO,Jia:2003:SAO,Levin:2003:PGA}. In the univariate case, the extreme cases are given by B-splines and the interpolatory $2n$-point schemes~\cite{Deslauriers:1989:SII}. While B-splines have the highest possible smoothness for a given support size but poor approximation order, the limit functions of the interpolatory schemes have optimal approximation order but low smoothness.

In the univariate binary setting, the generation and reproduction degrees are closely related to the behaviour of the subdivision symbol and its derivatives at $z=1$ and $z=-1$. For example, it is well known that a subdivision scheme with symbol $a(z)$ generates and reproduces constant functions, if $a(1)=2$ and $a(-1)=0$, which in turn is a necessary condition for the convergence of the scheme. Regarding higher degrees of generation and reproduction, Cavaretta et al.~\cite{Cavaretta:1991:SS} show that a convergent scheme \emph{generates} polynomials up to degree $\dG\ge1$ if $a^{(k)}(-1)=0$ for $k=0,\dots,\dG$, which is equivalent to the fact that the symbol can be written as $a(z)=(1+z)^{\dG+1}b(z)$. Moreover, Dyn et al.~\cite{Dyn:2008:PRB} prove that a convergent primal scheme further \emph{reproduces} polynomials up to degree $\dR\ge1$ with $\dR\le\dG$, if $a(z)=2+(1-z)^{\dR+1}c(z)$, or equivalently, $a^{(k)}(1)=0$ for $k=1,\dots,\dR$. Similar algebraic conditions exist in the bivariate setting, and we review them in Section~\ref{sec:preliminaries}.

Primal pseudo-splines with symbols
\begin{equation}\label{eq:univariate-pseudo-spline}
  u^l_n(z) = 2 {\sigma(z)}^n \sum_{i=0}^l \binom{n+i-1}{i} {\delta(z)}^i,\qquad
  0 \le l < n,
\end{equation}
where
\begin{equation}\label{eq:sigma-delta}
  \sigma(z) =  \frac{{(1+z)}^2}{4z},  \qquad\qquad
  \delta(z) = -\frac{{(1-z)}^2}{4z},
\end{equation}
were discovered by Dong and Shen~\cite{Dong:2007:PWA} and form a family of subdivision schemes that neatly fills the gap between the odd degree B-splines with symbols $u^0_n(z)$ and the interpolatory $2n$-point schemes by Deslauriers and Dubuc~\cite{Deslauriers:1989:SII} with symbols $u^{n-1}_n(z)$.
It follows directly from~\eqref{eq:univariate-pseudo-spline} that the pseudo-spline $u^l_n$ generates polynomials up to degree $2n-1$, and
\rev{Dyn et al.~\cite[Section~6]{Dyn:2008:PRB} show that $u^l_n(z)$} can be rewritten as
\begin{equation}\label{eq:univariate-pseudo-spline-alt}
  u^l_n(z) = 2 - 2 {\delta(z)}^{l+1} \sum_{i=1}^n \binom{n+l}{i+l} {\delta(z)}^{i-1}{\sigma(z)}^{n-i},
\end{equation}
\rev{from which it is straightforward to see}
that it reproduces polynomials up to degree $2l+1$. 
Moreover, $u^l_n(z)$ is the symbol with minimal support among all symmetric symbols that share these degrees of polynomial generation and reproduction.

The goal of this paper is to generalize the concept of univariate pseudo-splines to the bivariate setting. We provide explicit formulas for the symbols $a^l_n(\bz)$, $0\le l<n$ of symmetric four-directional bivariate pseudo-splines and prove that they satisfy the algebraic properties of polynomial generation and reproduction up to degree $2n-1$ and $2l+1$, respectively. The analysis of further properties like convergence, smoothness, stability, and non-singularity will be given in detail in a forthcoming paper.

After briefly discussing the special case of tensor product pseudo-splines in Section~\ref{sec:TP-pseudo-splines}, we first turn our attention to a reasonable four-directional bivariate analogue of univariate B-splines in Section~\ref{sec:box-splines} and show that the pseudo-splines with generation degree $2n-1$ and linear reproduction are scaled four-directional box-splines. In Section~\ref{sec:interpolatory}, we then consider the analogue of the univariate Dubuc--Delauriers schemes over the four-directional grid and provide an explicit formula for the minimally supported interpolatory schemes with generation and reproduction degree $2n-1$ that were discovered by Han and Jia~\cite{Han:1998:OIS}. We finally introduce and analyse a complete family of bivariate pseudo-splines in Section~\ref{sec:pseudo-splines}, provide some examples in Section~\ref{sec:examples}, and conclude with some observations concerning the uniqueness of this family in Section~\ref{sec:conclusion}.


\section{Preliminaries}\label{sec:preliminaries}
The \emph{symbol} of the bivariate subdivision scheme, defined by the finitely supported \emph{subdivision mask} $A=\{a_\balpha\in\RR:\balpha\in\ZZ^2\}$, is given by the Laurent polynomial
\[
  a(\bz) = \sum_{\balpha\in\ZZ^2} a_\balpha \bz^\balpha, \qquad
  \bz = (z_1,z_2) \in {(\CC\setminus\{0\})}^2,
\]
where $\bz^\balpha=z_1^{\alpha_1}z_2^{\alpha_2}$ for $\balpha=(\alpha_1,\alpha_2)\in\ZZ^2$. \rev{In the four-directional setting that we consider, the symbol is called \emph{symmetric} if
\[
  a(z_1,z_2) = a(1/z_1,z_2) = a(z_1,1/z_2) = a(z_2,z_1),
\]
or, in terms of the mask coefficients,
\[
  a_{(\alpha_1,\alpha_2)}
  = a_{(-\alpha_1,\alpha_2)}
  = a_{(\alpha_1,-\alpha_2)}
  = a_{(\alpha_2,\alpha_1)}.
\]
The \emph{support} of the mask, the symbol, and the scheme is defined as the convex hull of the set $\{ \balpha\in\ZZ_2 : a_{\balpha} \neq \bzero\}$, and the \emph{size} of the support is the area of this convex hull.}

Similar to the univariate setting, the generation and reproduction degrees are closely related to the behaviour of the symbol and its derivatives at $\bz\in E$, where $E=\{(1,1),(-1,1),(1,-1),(-1,-1)\}$. For example, the generation and reproduction of constant functions is guaranteed, if $a(1,1)=4$ and $a(\bz)=0$ for $\bz\in E'$, where $E'=E\setminus\{(1,1)\}$, which is again a necessary condition for the convergence of the scheme. Regarding higher degrees of generation and reproduction, Cavaretta et al.~\cite{Cavaretta:1991:SS} show that a convergent bivariate scheme \emph{generates} polynomials up to degree $\dG\ge1$, if
\begin{equation}\label{eq:PG-condition}
  (D^{\bk}a)(\bz) = 0, \qquad
  \bz \in E', \quad
  \bk \in \NN_0^2, \quad
  0 \le \abs{\bk} \le \dG,
\end{equation}
which is also known as the \emph{sum rule of order $\dG+1$}. Moreover, Charina et al.~\cite{Charina:2013:PRO} prove that a convergent primal scheme \rev{that generates polynomials up to degree $m$} further \emph{reproduces} polynomials up to degree $\dR\ge1$ with $\dR\le\dG$, if
\begin{equation}\label{eq:PR-condition}
  (D^{\bk}a)(1,1) = 0, \qquad
  \bk \in \NN_0^2, \quad
  0 < \abs{\bk} \le \dR.
\end{equation}

While the corresponding conditions in the univariate setting have equivalent \emph{factorization} properties, the bivariate analogue can be described in terms of ideals, leading to equivalent \emph{decomposition} properties. Sauer~\cite{Sauer:2002:PII} shows that
\begin{equation}\label{eq:J}
  \cJ_k = \bigideal{1-\bz^2}^k
        = \bigideal{{(1-z_1^2)}^{\alpha_1} {(1-z_2^2)}^{\alpha_2} : \balpha\in\NN_0^2, \abs{\balpha}=k}, \qquad
  k \ge 1,
\end{equation}
is the ideal of all bivariate polynomials which satisfy
\[
  (D^{\bk}a)(\bz) = 0, \qquad
  \bz \in E, \quad
  \bk \in \NN_0^2, \quad
  0 \le \abs{\bk} < k,
\]
and that the bivariate polynomials which satisfy only~\eqref{eq:PG-condition} for $\dG=k-1$ belong to the quotient ideal
\begin{equation}\label{eq:I}
  \cI_k = \cJ_k : \ideal{1-\bz}^k, \qquad
  k \ge 1.
\end{equation}
Consequently, a convergent scheme with symbol $a\in\cI_k$ generates polynomials up to degree $k-1$. However, $a\in\cJ_k$ does not imply polynomial reproduction of degree $k-1$, because $a(1,1)=0$ in this case, and hence such a scheme is not even convergent~\cite{Dyn:2002:SSI}. But if $a$ reproduces polynomials up to degree $k-1$ and $b\in\cJ_k$, then the reproduction degree of $a+b$ is also $k-1$. Note that the indices of our versions of $\cI_k$ in~\eqref{eq:I} and $\cJ_k$ in~\eqref{eq:J} are shifted by one with respect to those in~\cite{Sauer:2002:PII} for convenience, so that
\[
  a \in \cI_k, \quad b \in \cI_l \qquad\Longrightarrow\qquad a \cdot b \in \cI_{k+l},
\]
and similarly for $\cJ_k$.

Now, in order to check the generation degree of a symbol, we recall from Charina et al.~\cite{Charina:2011:SMS}, \rev{that the \emph{four-directional box-spline} ${\cal B}_{2i,2j,k,k}$ with symbol\footnote{\rev{Note that we use the notation $B_{i,j,k}$ for convenience, and it must not be confused with the symbol of the three-directional box-spline ${\cal B}_{i,j,k}$.}}}
\begin{equation}\label{eq:box-spline}
  B_{i,j,k}(\bz) =
    \biggl(\frac{(1+z_1)^2}{4z_1}\biggr)^i
    \biggl(\frac{(1+z_2)^2}{4z_2}\biggr)^j
    \biggl(\frac{(1+z_1z_2)(z_1+z_2)}{4z_1z_2}\biggr)^k
\end{equation}
is contained in $\cI_{2\dG}$, where $\dG=i+j+k-\max(i,j,k)$. Therefore, if the symbol of a scheme can be written as
\begin{equation}\label{eq:PG-decomposition}
  a(\bz) = \sum_{n=1}^N B_{i_n,j_n,k_n}(\bz) b_n(\bz)
\end{equation}
with $i_n+j_n+k_n-\max(i_n,j_n,k_n)\ge\dG$, $n=1,\dots,N$, for some suitable bivariate symbols $b_n$, then $a\in\cI_{2\dG}$ and hence the scheme generates polynomials up to degree $2\dG-1$. The scheme further reproduces polynomials up to degree $2\dR-1$ with $\dR\le\dG$, if the symbol can be decomposed as
\begin{equation}\label{eq:PR-decomposition}
  a(\bz) = 4 + \sum_{n=1}^{N'} {\delta(z_1)}^{\alpha_n} {\delta(z_2)}^{\beta_n} c_n(\bz)
\end{equation}
with $\alpha_n+\beta_n\ge\dR$, $n=1,\dots,N'$, for some suitable bivariate symbols $c_n$.

\begin{example}\label{example:cubic-reproduction-scheme}
The symmetric subdivision scheme with mask
\[
  A = \frac{1}{32}
    \begin{bmatrix}
    0 & 0 & -1 & -2 & -1 & 0 & 0\\
    0 & -2 & 0 & 4 & 0 & -2 & 0\\
    -1 & 0 & 10 & 18 & 10 & 0 & -1\\
    -2 & 4 & 18 & 24 & 18 & 4 & -2\\
    -1 & 0 & 10 & 18 & 10 & 0 & -1\\
    0 & -2 & 0 & 4 & 0 & -2 & 0\\
    0 & 0 & -1 & -2 & -1 & 0 & 0
    \end{bmatrix}
\]
and symbol
\begin{equation}\label{eq:example-scheme}
  a(\bz) = 12 B_{1,1,1}(\bz) - 8 B_{1,1,2}(\bz)
\end{equation}
generates polynomials up to degree $3$. It also reproduces polynomials up to the same degree, because
\[
  a(\bz) = 4 - 4 {\delta(z_1)}^2 ( B_{0,1,0} + 2 B_{1,1,0} )
             - 4 \delta(z_1) \delta(z_2) ( 1 + 4 B_{1,1,0} )
             - 4 {\delta(z_2)}^2 ( B_{1,0,0} + 2 B_{1,1,0} ).
\]
\end{example}

\begin{example}\label{example:cubic-reproduction-family}
Since the symbol
\[
  b(\bz) = B_{2,2,0}(\bz) - B_{1,1,1}(\bz) = \delta(z_1^2) \delta(z_2^2)/16
\]
is contained in $\cJ_4$ by~\eqref{eq:J}, and therefore satisfies~\eqref{eq:PG-condition} and~\eqref{eq:PR-condition} for $\dG=\dR=3$, we can add any multiple of it to the symbol in~\eqref{eq:example-scheme} without changing the generation and reproduction degree. Therefore, the symmetric subdivision scheme with mask
\[
  A_\mu = \frac{1}{32}
    \begin{bmatrix}
    0 & 0 & -1 & -2 & -1 & 0 & 0\\
    0 & \mu & 0 & -2\mu & 0 & \mu & 0\\
    -1 & 0 & 10 & 18 & 10 & 0 & -1\\
    -2 & -2\mu & 18 & 32+4\mu & 18 & -2\mu & -2\\
    -1 & 0 & 10 & 18 & 10 & 0 & -1\\
    0 & \mu & 0 & -2\mu & 0 & \mu & 0\\
    0 & 0 & -1 & -2 & -1 & 0 & 0
    \end{bmatrix}
\]
and symbol
\[
  a_\mu(\bz)
  = a(\bz) + 8(2+\mu) b(\bz)
  = 8(2+\mu) B_{2,2,0}(\bz) - 4(1+2\mu) B_{1,1,1}(\bz) - 8 B_{1,1,2}(\bz)
\]
generates and reproduces polynomials up to degree 3 for any $\mu\in\RR$.
\end{example}

\begin{remark}\label{remark:non-unique}
Besides illustrating the concepts of polynomial generation and reproduction, the examples above also show an important difference between the univariate and the bivariate four-directional setting. As we will see in Section~\ref{sec:interpolatory}, the scheme $a_\mu$ for $\mu=0$ is the four-directional analogue of the interpolatory 4-point scheme, and it was proven by Han and Jia~\cite{Han:1998:OIS} that its support is minimal \rev{(with respect to its size)}. However, while the interpolatory $2n$-point schemes are the \emph{unique} schemes with generation and reproduction degree $2n-1$ and minimal support in the univariate setting, Example~\ref{example:cubic-reproduction-family} shows that, at least in the case $n=2$, the four-directional setting admits a whole family of schemes $a_\mu$ which generate and reproduce polynomials up to degree $2n-1$ and are minimally supported. We shall come back to this observation in Section~\ref{sec:conclusion}.
\end{remark}

\subsection{Notation}\label{sec:notation}
In what follows it will be useful to define the bivariate analogues of $\sigma$ and $\delta$ in~\eqref{eq:sigma-delta}, their difference, and their product as
\[
  \bsigma(\bz) = \sigma(z_1) \sigma(z_2), \qquad
  \bdelta(\bz) = \delta(z_1) \delta(z_2), \qquad
  \bgamma(\bz) = \bsigma(\bz) - \bdelta(\bz), \qquad
  \bpi(\bz) = \bsigma(\bz) \bdelta(\bz).
\]
We further introduce the notation
\[
  {\bpi(\bz)}^\balpha = \bigl( \sigma(z_1) \delta(z_1) \bigr)^{\alpha_1}
                        \bigl( \sigma(z_2) \delta(z_2) \bigr)^{\alpha_2}
\]
and note that
\begin{equation}\label{eq:bpi^balpha}
  {\bpi(\bz)}^\balpha
  = \frac{{\delta(z_1^2)}^{\alpha_1} {\delta(z_2^2)}^{\alpha_2}}{{4}^{\alpha_1+\alpha_2}}
  \in \cJ_{2\abs{\balpha}} \subset \cI_{2\abs{\balpha}}.
\end{equation}

Besides the degrees of polynomial generation and reproduction, we are also interested in the \emph{support of a symbol}, and we frequently use the graphical notation
\[
  \support[2m+1]{2n+1}{l}
\]
to denote that a scheme is supported on the octagon
\[
  \{ \balpha :
      \abs{\alpha_1} \le m,
      \abs{\alpha_2} \le n,
      \abs{\alpha_1}+\abs{\alpha_2} \le m+n-l
  \},
\]
or rather on the rectangle $[-m,m]\times[-n,n]$, minus the triangular regions with side length $l$ in each corner. To keep the notation compact, we may omit the horizontal dimension if $m=n$. Following this convention, the support of the four-directional box-spline in~\eqref{eq:box-spline} is
\[
  \support[2(i+k)+1]{2(j+k)+1,}{k}
\]
and the support of the schemes in Examples~\ref{example:cubic-reproduction-scheme} and~\ref{example:cubic-reproduction-family} is
\[
  \support{7.}{2}
\]


\section{Special cases}
Before introducing and analysing the complete family of four-directional bivariate pseudo-splines in Section~\ref{sec:pseudo-splines}, let us review some special cases and summarize their properties.

\subsection{Tensor product pseudo-splines}\label{sec:TP-pseudo-splines}
The simplest approach to constructing symbols of potential bivariate pseudo-splines is to consider the tensor product of univariate pseudo-splines.

\begin{proposition}\label{proposition:TP-pseudo-spline}
The bivariate symbols
\[
  \bar{a}^l_n (\bz) = u^l_n(z_1) u^l_n(z_2), \qquad
  0\le l<n
\]
generate polynomials up to degree $2n-1$ and reproduce polynomials up to degree $2l+1$.
\end{proposition}

\begin{proof}
By~\eqref{eq:univariate-pseudo-spline} we have
\[
  \bar{a}^l_n (\bz)
  = 4 {\bsigma(\bz)}^n \sum_{i=0}^l \sum_{j=0}^l \binom{n+i-1}{i} \binom{n+j-1}{j} {\delta(z_1)}^i {\delta(z_2)}^j,
\]
which explains the degree of polynomial generation by~\eqref{eq:PG-decomposition}, because ${\bsigma(\bz)}^n=B_{n,n,0}(\bz)$. The degree of polynomial reproduction follows from~\eqref{eq:PR-decomposition} once we use~\eqref{eq:univariate-pseudo-spline-alt} to obtain
\[
  \bar{a}^l_n (\bz)
  = 4 - 4 {\delta(z_1)}^{l+1} v^l_n(z_1)
      - 4 {\delta(z_1)}^{l+1} {\delta(z_2)}^{l+1} v^l_n(z_1) v^l_n(z_2)
      - 4 {\delta(z_2)}^{l+1} v^l_n(z_2),
\]
where
\begin{equation}\label{eq:v^l_n}
  v^l_n(z) = \sum_{i=1}^n \binom{n+l}{i+l} {\delta(z)}^{i-1}{\sigma(z)}^{n-i}.
\end{equation}
\end{proof}

\noindent
However, the support of these schemes is
\[
  \support{2n+2l+1,}{0}
\]
which is certainly not minimal as we will see in the following sections. This fact was already observed by Han and Jia~\cite{Han:1998:OIS} for the special case of interpolatory schemes with symbols $\bar{a}^{n-1}_n(\bz)$, and we revisit their idea of constructing minimally supported interpolatory schemes in Section~\ref{sec:interpolatory}.


\subsection{Four directional box-splines}\label{sec:box-splines}
Let us now turn to the case of four-directional bivariate pseudo-splines with generation degree $2n-1$, $n\ge1$ and linear reproduction. These schemes turn out to be scaled box-splines and can be considered a four-directional bivariate analogue of univariate B-splines.

\begin{proposition}\label{proposition:box-spline}
The bivariate symbols
\begin{equation}\label{eq:a-box}
  \tilde{a}_n (\bz)
  = 4 B_{\ceil{n/2}, \ceil{n/2}, \floor{n/2}} (\bz)
  = 4 {\bsigma(\bz)}^{\ceil{n/2}} {\bgamma(\bz)}^{\floor{n/2}}, \qquad
  n \ge 1
\end{equation}
generate polynomials up to degree $2n-1$ and reproduce polynomials up to degree $1$.
\end{proposition}

\begin{proof}
Using the general dimension formula for box-splines~\eqref{eq:PG-decomposition} and noticing that $\ceil{n/2}\ge \floor{n/2}$ and $\ceil{n/2}+\floor{n/2}=n$, it is clear that $\tilde{a}_n\in\cI_{2n}$, which explains the degree of polynomial generation. To see the degree of polynomial reproduction, we observe that
\[
  \sigma'(z) = \frac{z^2-1}{4z^2},\qquad
  \delta'(z) = \frac{1-z^2}{4z^2},\qquad
\]
hence $\sigma'(1) = \delta'(1) = 0$, and it follows that the first partial derivatives of $\bsigma(\bz)$ and $\bgamma(\bz)$ are zero at $(1,1)$.
\end{proof}

\noindent
From the considerations at the end of Section~\ref{sec:notation} we conclude that the support of $\tilde{a}_n$ is
\[
  \support{2n+1,}{\floor{n/2}}
\]
which is smaller than the support of the tensor product symbol $\bar{a}_n^0$ with the same degrees of generation and reproduction. Our numerical experiments make us believe that this is actually the smallest possible support, but we do not have a proof.


\subsection{Interpolatory schemes}\label{sec:interpolatory}
The four-directional bivariate generalization of the interpolatory $2n$-point schemes was first mentioned and analysed by Han and Jia~\cite{Han:1998:OIS}. They show that there exists a unique symmetric interpolatory scheme with generation degree $2n-1$ and minimal support
\begin{equation}\label{eq:Han-Jia-support}
  \support{4n-1,}{2n-2}
\end{equation}
but they do not give an explicit formula for the symbols of these schemes. We discovered that these symbols can be represented nicely in terms of the symbols of the univariate $2n$-point schemes.

\begin{proposition}\label{proposition:Han-Jia-schemes}
The bivariate symbols
\begin{equation}\label{eq:Han-Jia-scheme}
  \hat{a}_n (\bz) = \sum_{i=0}^{n-1} u_{n-i}^{n-i-1}  (z_1) u_{i+1}^i(z_2)
                  - \sum_{i=0}^{n-2} u_{n-i-1}^{n-i-2}(z_1) u_{i+1}^i(z_2), \qquad
  n \ge 1
\end{equation}
are interpolatory, generate and reproduce polynomials up to degree $2n-1$, and are minimally supported.
\end{proposition}

\begin{proof}
The univariate $2n$-point schemes satisfy $u^{n-1}_n(z)+u^{n-1}_n(-z)=2$, because they are interpolatory~\cite{Dyn:2002:SSI}. Consequently,
\begin{multline*}
  \hat{a}_n(z_1,z_2) + \hat{a}_n(z_1,-z_2) + \hat{a}_n(-z_1,z_2) + \hat{a}_n(-z_1,-z_2)\\
  \begin{aligned}
    &= \sum_{i=0}^{n-1} u_{n-i}^{n-i-1}(z_1) \Bigl( u_{i+1}^i(z_2) + u_{i+1}^i(-z_2) \Bigr)
     + \sum_{i=0}^{n-1} u_{n-i}^{n-i-1}(-z_1) \Bigl( u_{i+1}^i(z_2) + u_{i+1}^i(-z_2) \Bigr)\\
    &\quad- \sum_{i=0}^{n-2} u_{n-i-1}^{n-i-2}(z_1) \Bigl( u_{i+1}^i(z_2) + u_{i+1}^i(-z_2) \Bigr)
     - \sum_{i=0}^{n-2} u_{n-i-1}^{n-i-2}(-z_1) \Bigl( u_{i+1}^i(z_2) + u_{i+1}^i(-z_2) \Bigr)\\
    &= 2 \sum_{i=0}^{n-1} \Bigl( u_{n-i}^{n-i-1}(z_1) + u_{n-i}^{n-i-1}(-z_1) \Bigr)
     - 2 \sum_{i=0}^{n-2} \Bigl( u_{n-i-1}^{n-i-2}(z_1) + u_{n-i-1}^{n-i-2}(-z_1) \Bigr)\\
    &= 4n - 4(n-1) = 4,
  \end{aligned}
\end{multline*}
which implies that the schemes $\hat{a}_n$ are interpolatory, too~\cite{Conti:2013:ACA}. Therefore, the degrees of polynomial generation and reproduction are the same~\cite[Proposition~3.4]{Charina:2013:PRO}, \rev{and the generation degree follows as a special case from Theorem~\ref{theorem:PG-PR} for $l=n-1$.} As for the support, we remember that $u^{n-1}_n$ is supported on $[-2n+1,2n-1]$. Hence, the supports of the symbols in the first sum in~\eqref{eq:Han-Jia-scheme} are rectangular and add up like
\[
  \support[4n-1]{3\qquad+}{0} \qquad\qquad\quad
  \support[4n-5]{7\qquad+\quad\cdots\quad+}{0} \qquad\qquad\qquad\qquad\quad
  \support[3]{4n-1\qquad=}{0} \qquad\qquad\qquad\qquad\quad
  \support[4n-1]{4n-1.}{2n-2}\qquad\phantom{,}
\]
Similarly, the supports of the symbols in the second sum add up to
\[
  \support[4n-5]{4n-5,}{2n-4}
\]
which is contained in the support form the first sum, so that the support of $\hat{a}_n$ matches the minimal support in~\eqref{eq:Han-Jia-support}.
\end{proof}


\section{A family of \rev{symmetric} four-directional bivariate pseudo-splines}\label{sec:pseudo-splines}
We are now ready to propose a whole family of four-directional bivariate pseudo-splines.

\begin{definition}
For any $n\ge1$ and $0\le l<n$, let
\begin{equation}\label{eq:family}
  a_n^l(\bz) = \sum_{i=0}^l \tilde{a}_{n-i}(\bz) b_n^i(\bz),
  \qquad
  b_n^i(\bz) = \sum_{j=0}^i c_n^{(i,j)} {\bpi(\bz)}^{(i-j,j)},
\end{equation}
with \rev{$\tilde{a}_n(\bz)$ and $\bpi(\bz)^{\balpha}$ as in~\eqref{eq:a-box} and~\eqref{eq:bpi^balpha}, respectively, and} real coefficients
\begin{equation}\label{eq:coefficients}
  c_n^{(i,j)} = \sum_{k=0}^{\floor{\frac{i}{2}}}
                  \binom{\floor{\frac{n-i}{2}}+k-1}{k}
                  \binom{n+i-2j-1}{i-j-k}
                  \binom{n+2j-i-1}{j-k},
  \qquad 0 \le j \le i < n.
\end{equation}
\end{definition}

\noindent
This family generalizes the two special cases in Sections~\ref{sec:box-splines} and~\ref{sec:interpolatory} and gives four-directional bivariate subdivision schemes with generation degree $2n-1$ and reproduction degree $2l+1$.

\begin{proposition}\label{proposition:generalized-box}
The symbols in~\eqref{eq:family} generalize the symbols in~\eqref{eq:a-box}, because
\[
  a_n^0 (\bz) = \tilde{a}_n (\bz),
  \qquad n\ge1.
\]
\end{proposition}

\begin{proof}
The statement follows immediately by noting that $c_n^{(0,0)} = 1$ for $n\ge1$.
\end{proof}

\begin{proposition}\label{proposition:generalized-Han-Jia}
The symbols in~\eqref{eq:family} generalize the symbols in~\eqref{eq:Han-Jia-scheme}, because
\[
  a_n^{n-1} (\bz) = \hat{a}_n (\bz),
  \qquad n\ge1.
\]
\end{proposition}

\begin{proof}
The proof consists of two main steps. We start by showing that
\[
  \hat{a}_n(\bz) = 4 \bsigma (\bz)\biggl( d_n^{n-1}(\bz) + \bgamma(\bz) \sum_{k=0}^{n-2} \bsigma^{n-2-k}(\bz) d_n^k(\bz) \biggr),
\]
where
\begin{equation}\label{eq:d^l_n-new}
  d^l_n (\bz)
  = \sum_{j=0}^l \binom{n+l-2j-1}{l-j} \binom{n+2j-l-1}{j} {\bpi(\bz)}^{(l-j,j)}, \qquad
  0 \le l < n.
\end{equation}
To this purpose we first derive two helpful identities for the univariate pseudo-splines $u_n^l$. On the one hand, it follows directly from~\eqref{eq:univariate-pseudo-spline} that
\begin{equation}\label{eq:u-identity-1}
  u_n^l(z) - u_n^{l-1}(z) = 2 {\sigma(z)}^n \binom{n+l-1}{l} {\delta(z)}^l,
\end{equation}
and, on the other hand, we have
\begin{align}\label{eq:u-identity-2}\notag
  u_n^l(z) - u_{n-1}^l(z)
  &= 2 {\sigma(z)}^{n-1} \Biggl[ (1-\delta(z)) \sum_{i=0}^l \binom{n+i-1}{i} {\delta(z)}^i
                                             - \sum_{i=0}^l \binom{n+i-2}{i} {\delta(z)}^i
                         \Biggr]\\\notag
  &= 2 {\sigma(z)}^{n-1} \Biggl[ \sum_{i=0}^l \binom{n+i-1}{i} {\delta(z)}^i
                               - \sum_{i=1}^{l+1} \binom{n+i-2}{i-1} {\delta(z)}^i
                               - \sum_{i=0}^l \binom{n+i-2}{i} {\delta(z)}^i
                         \Biggr]\\\notag
  &= 2 {\sigma(z)}^{n-1} \Biggl[ \sum_{i=0}^l \binom{n+i-2}{i-1} {\delta(z)}^i
                               - \sum_{i=1}^{l+1} \binom{n+i-2}{i-1} {\delta(z)}^i
                         \Biggr]\\
  &= -2 {\sigma(z)}^{n-1} \binom{n+l-1}{l} {\delta(z)}^{l+1},
\end{align}
where both identities hold for $0\le l<n$, if we extend the definition of $u_n^l$ by letting
\[
  u_0^0(z) = 2
  \qquad\qquad\text{and}\qquad\qquad
  u_n^{-1}(z) = 0, \qquad n > 0.
\]
We then conclude from~\eqref{eq:u-identity-1} that
\[
  \sum_{j=0}^l \bigl[ u_{n-j}^{l-j}(z_1) - u_{n-j}^{l-j-1}(z_1) \bigr]
               \bigl[ u_{n-l+j}^{j}(z_2) - u_{n-l+j}^{j-1}(z_2) \bigr]
  = 4 {\bsigma(\bz)}^{n-l} d_n^l(\bz)
\]
and from~\eqref{eq:u-identity-2} that
\[
  \sum_{j=0}^{l-1} \bigl[ u_{n-j}^{l-1-j}(z_1) - u_{n-j-1}^{l-1-j}(z_1) \bigr]
                   \bigl[ u_{n-l+j+1}^{j}(z_2) - u_{n-l+j}^{j}    (z_2) \bigr]
  = 4 {\bsigma(\bz)}^{n-l} \bdelta(\bz) d_n^{l-1}(\bz).
\]
By letting
\[
  e_n^l(\bz) = \sum_{j=0}^l u_{n-j}^{l-j}(z_1) u_{n-l+j}^{j}(z_2)
\]
and omitting the argument $(\bz)$ for the sake of brevity, we get
\[
  4 \bsigma^{n-l} (d_n^l - \bdelta d_n^{l-1}) = e_n^l - e_{n-1}^{l-1} - e_n^{l-1} + e_{n-1}^{l-2}
\]
and further
\begin{align*}
  \hat{a}_n
  &= \sum_{k=0}^{n-1} (e_n^k - e_{n-1}^{k-1})
   - \sum_{k=0}^{n-2} (e_n^k - e_{n-1}^{k-1})\\
  &= 4 \bsigma^n
   + \sum_{k=1}^{n-1} (e_n^k - e_{n-1}^{k-1})
   - \sum_{k=1}^{n-1} (e_n^{k-1} - e_{n-1}^{k-2})\\
  &= 4 \bsigma^n
   + 4 \sum_{k=1}^{n-1} \bsigma^{n-k} (d_n^k - \bdelta d_n^{k-1})\\
  &= 4 \bsigma \biggl( d_n^{n-1} + \bgamma \sum_{k=0}^{n-2} \bsigma^{n-k-2} d_n^k \biggr).
\end{align*}

In the second step, we now prove that this representation of $\hat{a}_n$ is identical to the formula of $a_n^{n-1}$ in~\eqref{eq:family}. To this end, we first observe that
\begin{align*}
  \hat{a}_n - 4\bsigma d_n^{n-1}
  &= 4\bsigma\bgamma \sum_{i=0}^{n-2} \bsigma^i d_n^{n-2-i} \\
  &= 4\bsigma\bgamma \sum_{i=0}^{n-2} \bsigma^{\ceil{\frac{i}2}}
                                      (\bgamma + \bdelta)^{\floor{\frac{i}2}} d_n^{n-2-i}\\
  &= 4\bsigma\bgamma \sum_{i=0}^{n-2} \bsigma^{\ceil{\frac{i}2}}
      \sum_{k=0}^{\floor{\frac{i}2}} \binom{\floor{\frac{i}2}}{k} \bgamma^{\floor{\frac{i}2}-k} \bdelta^k
      d_n^{n-2-i}\\
  &= 4\sum_{i=0}^{n-2} \sum_{k=0}^{\floor{\frac{i}2}} \binom{\floor{\frac{i}2}}{k}
       \bsigma^{\ceil{\frac{i}2}-k+1} \bgamma^{\floor{\frac{i}2}-k+1} \bpi^{(k,k)}
       d_n^{n-2-i}.
\end{align*}
We now rearrange the summation order, substitute $(i,k)$ with $(i+2k,k)$, and use the fact that
$\tilde{a}_{i+2}=4\bsigma^{\ceil{\frac{i}{2}}+1}\bgamma^{\floor{\frac{i}{2}}+1}$ to get
\begin{align*}
  \hat{a}_n - 4\bsigma d_n^{n-1}
  &= \sum_{i=0}^{n-2} \tilde{a}_{i+2} \sum_{k=0}^{\floor{\frac{n-i}2}-1}
       \binom{\floor{\frac{i}2}+k}{k} \bpi^{(k,k)} d_n^{n-2-i-2k}\\
  &= \sum_{i=0}^{n-2} \tilde{a}_{n-i} \sum_{k=0}^{\floor{\frac{i}2}}
       \binom{\floor{\frac{n-i}2}+k-1}{k} \bpi^{(k,k)} d_n^{i-2k}.
\end{align*}
Substituting $d_n^{i-2k}$ according to~\eqref{eq:d^l_n-new}, we then have
\begin{align*}
  \hat{a}_n - 4\bsigma d_n^{n-1}
  &= \sum_{i=0}^{n-2} \tilde{a}_{n-i} \sum_{k=0}^{\floor{\frac{i}2}} \binom{\floor{\frac{n-i}2}+k-1}{k}
       \sum_{j=0}^{i-2k} \binom{n+i-2k-2j-1}{i-2k-j} \binom{n+2j-i+2k-1}{j} \bpi^{(i-k-j,j+k)}\\
  &= \sum_{i=0}^{n-2} \tilde{a}_{n-i} \sum_{k=0}^{\floor{\frac{i}2}} \binom{\floor{\frac{n-i}2}+k-1}{k}
       \sum_{j=k}^{i-k} \binom{n+i-2j-1}{i-j-k} \binom{n+2j-i-1}{j-k} \bpi^{(i-j,j)},
\end{align*}
and noticing that the last sum does not change if we let $j$ range from $0$ to $i$, because the first binomial coefficient vanishes for $j>i-k$ and the second vanishes for $j<k$, we conclude that
\[
  \hat{a}_n - 4\bsigma d_n^{n-1}
  = \sum_{i=0}^{n-2} \tilde{a}_{n-i} \sum_{j=0}^i c_n^{(i,j)} \bpi^{(i-j,j)}.
\]
To complete the proof, we finally observe that
\[
  4\bsigma d_n^{n-1}
  = \tilde{a}_1 \sum_{j=1}^{n-1} \binom{2n-2-2j}{n-1-j} \binom{2j}{j} \bpi^{(n-1-j,j)}
  = \tilde{a}_1 \sum_{j=1}^{n-1} c_n^{(n-1,j)} \bpi^{(n-1-j,j)},
\]
because for $i=n-1$, the first binomial coefficient in~\eqref{eq:coefficients} is $1$ for $k=0$ and $0$ for $k>0$.
\end{proof}

\begin{theorem}\label{theorem:PG-PR}
The symbols in~\eqref{eq:family} \rev{are symmetric}, generate polynomials up to degree $2n-1$, reproduce polynomials up to degree $2l+1$, and the symbols corresponding to $l=n-1$ are interpolatory.
\end{theorem}

\begin{proof}
\rev{We first observe that $\bsigma$, $\bdelta$, and $\bgamma$, and therefore $\tilde{a}_n$ are symmetric. Moreover, since
\[
  {\bpi(z_1,z_2)}^{(\alpha_1,\alpha_2)} = {\bpi(z_2,z_1)}^{(\alpha_2,\alpha_1)}
\]
and
\[
  c_n^{(i,j)} = c_n^{(i,i-j)},
  \qquad j=0,\dots,i,
\]
we conclude that $b_n^i$ is symmetric, hence also $a_n^l$.}
Next, we note that $a_n^{n-1}$ is interpolatory by Propositions~\ref{proposition:Han-Jia-schemes} and~\ref{proposition:generalized-Han-Jia}. Regarding the generation degree, we know from the proof of Proposition~\ref{proposition:box-spline} that $\tilde{a}_{n-i}\in\cI_{2n-2i}$ and it follows from~\eqref{eq:bpi^balpha} that $b_n^i\in\cI_{2i}$, because ${\bpi(\bz)}^{(i-j,j)}\in\cJ_{2i}\subset\cI_{2i}$, for $j=0,\dots,i$. Altogether, we thus get $a_n^l\in\cI_{2n}$. To prove the reproduction degree, we first note that $a_n^{n-1}$ reproduces polynomials up to degree $2n-1$, because the degrees of generation and reproduction coincide for interpolatory schemes~\cite{Charina:2013:PRO}. Hence,
\[
  D^{\balpha} a_n^{n-1}(1,1) = 0,
  \qquad 0 < \abs{\balpha} < 2n.
\]
Then, using the recursion
\[
  a_n^{l-1}(\bz) = a_n^l(\bz) - a_{n-l}^0(\bz) b_n^l(\bz),
  \qquad 0 < l < n,
\]
which follows directly from~\eqref{eq:family}, we conclude by induction that
\begin{align*}
  D^{\balpha} a_n^{l-1}(1,1)
  &= D^{\balpha} a_n^l(1,1) - D^{\balpha} \bigl( a_{n-l}^0(1,1) b_n^l(1,1) \bigr)\\
  &= D^{\balpha} a_n^l(1,1)
     - \sum_{\bbeta \le \balpha} D^{\balpha-\bbeta} a_{n-l}^0(1,1) D^{\bbeta} b_n^l(1,1)
   = 0,
  \qquad 0<\abs{\balpha}<2l,
\end{align*}
because
\[
  D^{\bbeta} b_n^l(1,1) = 0,  \quad  \hbox{for} \qquad 0<\abs{\bbeta}<2l,
\]
In fact, since the $\beta$-th derivative of ${(\sigma(z)\delta(z))}^l$ vanishes at $z=1$ for $\beta\le 2l-1$, we see that
\[
  D^{\bbeta} \bpi(\bz)^{(l-j,j)} \big|_{\bz=(1,1)} = 0
  \qquad\text{for}\quad \beta_1\le 2(l-j)-1 \quad\text{or}\quad \beta_2\le 2j-1.
\]
Therefore,
\[
  D^{\bbeta} b_n^l(\bz)
  = \sum_{j=0}^l c_n^{(l,j)} D^{\bbeta} {\bpi(\bz)}^{(l-j,j)}
\]
can be different from zero at $\bz=(1,1)$ only if $\beta_1\ge 2(l-j)$ and $\beta_2\ge 2j$, that is, for $\abs{\bbeta}\ge 2l$.
\end{proof}

\begin{table}[t!]\centering\small\renewcommand*\arraystretch{1.3}
  \def\vs{\setlength{\unitlength}{2.5ex}\parbox[c]{0pt}{\begin{picture}(0,4)(0,0)\end{picture}}}
  \begin{tabular}{r|r||ccccc}
    \multicolumn{2}{c||}{} & \multicolumn{5}{c}{$l$}\\\cline{3-7}
    \multicolumn{2}{c||}{} & 0 & 1 & 2 & 3 & 4\\\hline\hline
    & 1\vs & \support{ 3}{0} \\
    & 2\vs & \support{ 5}{1} & \support{ 7}{2} \\
$n$ & 3\vs & \support{ 7}{1} & \support{ 9}{3} & \support{11}{4}\\
    & 4\vs & \support{ 9}{2} & \support{11}{3} & \support{13}{5} & \support{15}{6}\\
    & 5\vs & \support{11}{2} & \support{13}{4} & \support{15}{5} & \support{17}{7} & \support{19}{8}\\
  \end{tabular}
  \caption{Support of the pseudo-splines $a^l_n$ with generation degree $2n-1$ and reproduction degree $2l+1$.}
  \label{table:support-a}
\end{table}

\begin{proposition}\label{proposition:support}
The support of the symbols in~\eqref{eq:family} is
\[
  \support{2(n+l)+1,\qquad\qquad0\le l < n.}{n+l-\bigceil{\tfrac{n-l}{2}}}
\]
\end{proposition}

\begin{proof}
As in the proof of Proposition~\ref{proposition:Han-Jia-schemes}, we notice that the supports of the symbols in the sum of $b_n^i$ are rectangular and add up like
\[
  \support[4i+1]{1\qquad+}{0} \qquad\qquad\quad
  \support[4i-3]{5\qquad+\quad\cdots\quad+}{0} \qquad\qquad\qquad\qquad\quad
  \support[1]{4i+1\qquad=}{0} \qquad\qquad\qquad
  \support[4i+1]{4i+1}{2i}\qquad\phantom{,}
\]
to a diamond-shaped domain with vertices $(\pm2i,0)$ and $(0,\pm2i)$. In view of Proposition~\ref{proposition:box-spline}, the support of $\tilde{a}_{n-i}$ is
\[
  \support{2(n-i)+1.}{\bigfloor{\tfrac{n-i}{2}}}
\]
Hence, the support of $\tilde{a}_{n-i} b^i_n$ is
\[
  \support{2(n+i)+1,}{n+i-\bigceil{\tfrac{n-i}{2}}}
\]
which is contained in the support of $\tilde a^0_{n-l} b^l_n$ for $i\le l$.
\end{proof}

\noindent
Table~\ref{table:support-a} shows the support of the pseudo-splines in~\eqref{eq:family} for $1\le n \le 5$. Note that Proposition~\ref{proposition:support} does not actually show that these pseudo-splines are \emph{minimally} supported, but we believe they are, and we actually verified this numerically for $n\le 20$.

\begin{proposition}\label{proposition:necessary}
The symbols in~\eqref{eq:family} satisfy the necessary conditions for convergence,
\[
  a_n^l(1,1) = 4,
  \qquad\qquad
  a_n^l(\bz) =0, \quad \bz \in E',
  \qquad\qquad 0\le l<n.
\]
\end{proposition}

\begin{proof}
The proof relies on a key property of box splines, whose symbols satisfy
\[
  \tilde{a}_n(1,1) = 4,
  \qquad\qquad
  \tilde{a}_n(\bz) =0, \quad \bz \in E',
  \qquad\qquad 0\le l<n.
\]
From this property we conclude that
\[
  a_n^l(1,1) = 4\sum_{i=0}^l b_n^i(1,1),
  \qquad\qquad
  a_n^l(\bz)=0\quad \bz \in E',
  \qquad\qquad 0\le l<n,
\]
and using the fact that ${(\sigma(z)\delta(z))}^i$ vanishes at $z=1$ for $i>0$, we get
\[
  \sum_{i=0}^l b_n^i(1,1)=b_n^{0}(1,1)=c_n^{(0,0)}=1,\qquad 0\le l<n.
\]
\end{proof}


\section{Examples}\label{sec:examples}
We now present some examples of the subdivision masks $A_n^l$ associated with the symbols $a_n^l$ in~\eqref{eq:family} for $1\le n\le 3$ and show the graphs of the corresponding basic limit functions obtained after three subdivision steps.

\begin{figure}[t!]
  \parbox{7cm}{\centering%
    $A_1^0 = \frac{1}{4}
    \begin{bmatrix}
    1 & 2 & 1\\
    2 & 4 & 2\\
    1 & 2 & 1
    \end{bmatrix}$}\hfill
  \parbox{8cm}{\centering\includegraphics{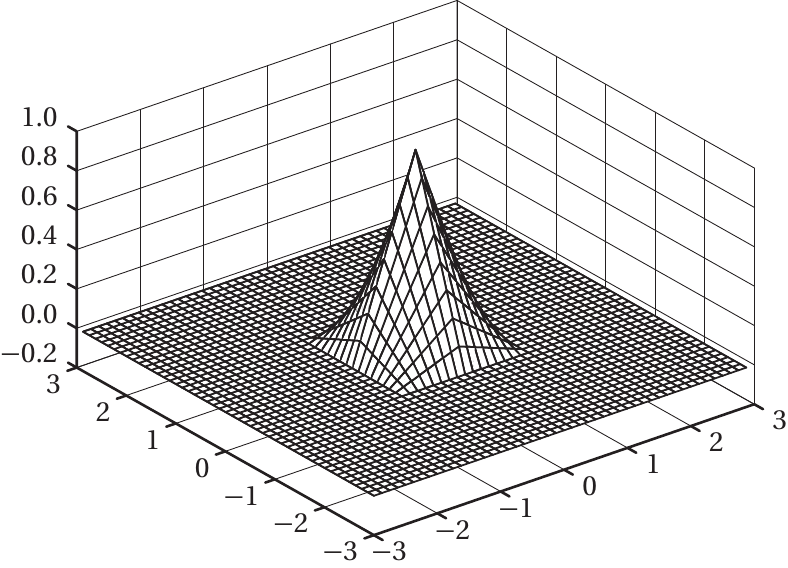}}
  \caption{Mask and graph of the basic limit function for the pseudo-spline with $n=1$.}
  \label{fig:example-n=1}
\end{figure}

\begin{figure}[t!]
  \parbox{7cm}{\centering\scalebox{0.85}{%
    $A_2^0 = \frac{1}{16}
    \begin{bmatrix}
    0 & 1 & 2 & 1 & 0\\
    1 & 4 & 6 & 4 & 1\\
    2 & 6 & 8 & 6 & 2\\
    1 & 4 & 6 & 4 & 1\\
    0 & 1 & 2 & 1 & 0
    \end{bmatrix}$}}\hfill
  \parbox{8cm}{\centering\includegraphics{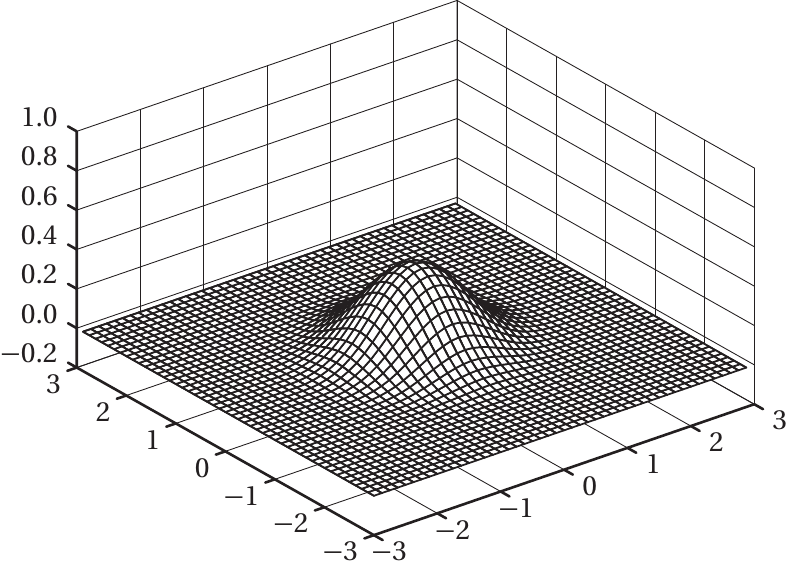}}\\[2ex]
  \parbox{7cm}{\centering\scalebox{0.85}{%
    $A_2^1 = \frac{1}{32}
    \begin{bmatrix}
     0 & 0 & -1 & -2 & -1 & 0 &  0\\
     0 & 0 &  0 &  0 &  0 & 0 &  0\\
    -1 & 0 & 10 & 18 & 10 & 0 & -1\\
    -2 & 0 & 18 & 32 & 18 & 0 & -2\\
    -1 & 0 & 10 & 18 & 10 & 0 & -1\\
     0 & 0 &  0 &  0 &  0 & 0 &  0\\
     0 & 0 & -1 & -2 & -1 & 0 &  0
    \end{bmatrix}$}}\hfill
  \parbox{8cm}{\centering\includegraphics{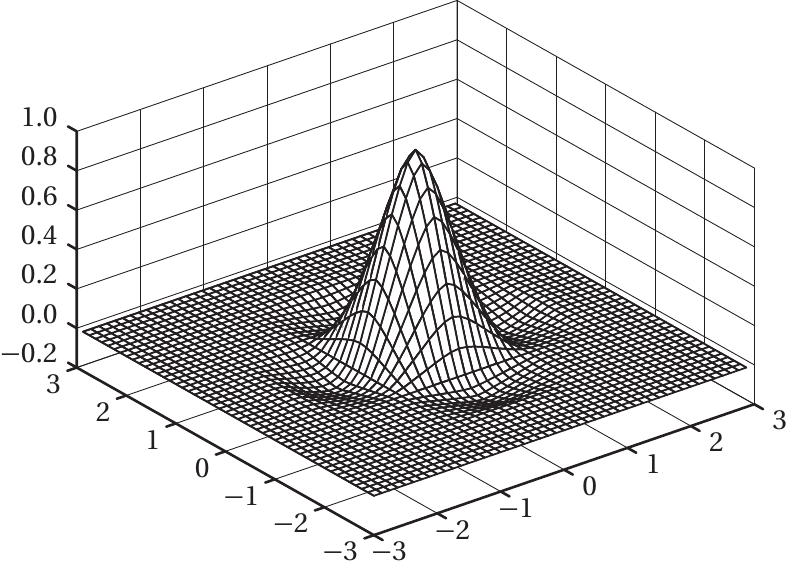}}\par
  \caption{Masks and graphs of the basic limit functions for the pseudo-splines with $n=2$.}
  \label{fig:example-n=2}
\end{figure}

For $n=1$, the only member of our pseudo-spline family is the tensor product pseudo-spline with symbol $a_1^0=\bar{a}_1^0=4B_{1,1,0}$ and polynomial generation and reproduction degrees $1$. It is an interpolatory scheme, and the limit function is piecewise bilinear (see Figure~\ref{fig:example-n=1}). For $n=2$, our family contains two members, which both generate polynomials up to degree~$3$ and reproduce polynomials up to degree $1$ and $3$, respectively. The first scheme is the four-directional box spline with symbol $a_2^0=\tilde{a}_2=4B_{1,1,1}$, and the second scheme with symbol $a_2^1$ is the four-directional bivariate analogue of the interpolatory 4-point scheme (see Figure~\ref{fig:example-n=2}). For $n=3$, our family contains the four-directional box spline with symbol $a_3^0=\tilde{a}_3=4B_{2,2,1}$ and the four-directional bivariate analogue of the interpolatory 6-point scheme with symbol $a_3^2$. Both schemes have polynomial generation degree $5$ and reproduction degrees $1$ and $5$, respectively. The third family member with symbol $a_3^1$ also generates polynomials up to degree $5$ and reproduces polynomials up to degree $3$. It fills the gap between the special cases $a_3^0$ and $a_3^2$ not only regarding the reproduction degree, but also regarding the support (see Table~\ref{table:support-a}) and the shape of the basic limit function (see Figure~\ref{fig:example-n=3}).

\begin{figure}[t!]
  \parbox{7cm}{\centering\scalebox{0.725}{%
    $A_3^0 = \frac{1}{256}
    \begin{bmatrix}
    0 &  1 &  4 &  6 &  4 &  1 & 0\\
    1 &  8 & 23 & 32 & 23 &  8 & 1\\
    4 & 23 & 56 & 74 & 56 & 23 & 4\\
    6 & 32 & 74 & 96 & 74 & 32 & 6\\
    4 & 23 & 56 & 74 & 56 & 23 & 4\\
    1 &  8 & 23 & 32 & 23 &  8 & 1\\
    0 &  1 &  4 &  6 &  4 &  1 & 0
    \end{bmatrix}$}}\hfill
  \parbox{8cm}{\centering\includegraphics{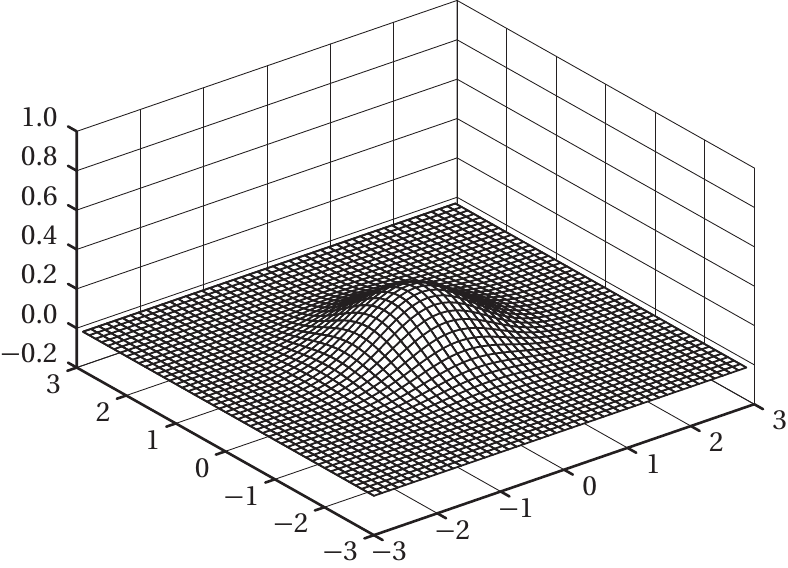}}\\[2ex]
  \parbox{7cm}{\centering\scalebox{0.725}{%
    $A_3^1 = \frac{1}{256}
    \begin{bmatrix}
     0 &  0  &  0 &  -3 & -6  &  -3 &  0 &  0  &  0\\
     0 &  0  & -2 &  -8 & -12 &  -8 & -2 &  0  &  0\\
     0 & -2  & -4 &  14 &  32 &  14 & -4 & -2  &  0\\
    -3 & -8  & 14 &  80 & 122 &  80 & 14 & -8  & -3\\
    -6 & -12 & 32 & 122 & 168 & 122 & 32 & -12 & -6\\
    -3 & -8  & 14 &  80 & 122 &  80 & 14 & -8  & -3\\
     0 & -2  & -4 &  14 &  32 &  14 & -4 & -2  &  0\\
     0 &  0  & -2 &  -8 & -12 &  -8 & -2 &  0  &  0\\
     0 &  0  &  0 &  -3 & -6  &  -3 &  0 &  0  &  0
    \end{bmatrix}$}}\hfill
  \parbox{8cm}{\centering\includegraphics{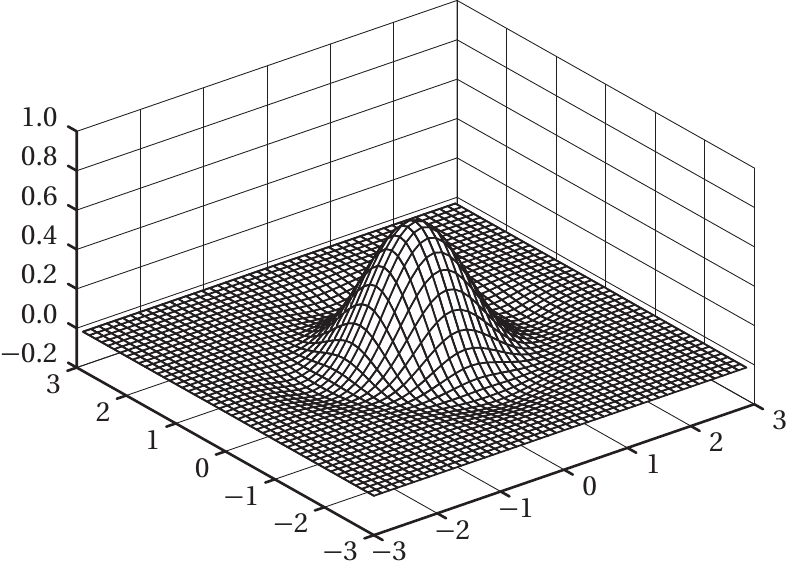}}\\[2ex]
  \parbox{7cm}{\centering\scalebox{0.725}{%
    $A_3^2 = \frac{1}{512}
    \begin{bmatrix}
    0 & 0 &  0  & 0 &  3  &  6  &  3  & 0 &  0  & 0 & 0\\
    0 & 0 &  0  & 0 &  0  &  0  &  0  & 0 &  0  & 0 & 0\\
    0 & 0 &  2  & 0 & -27 & -50 & -27 & 0 &  2  & 0 & 0\\
    0 & 0 &  0  & 0 &  0  &  0  &  0  & 0 &  0  & 0 & 0\\
    3 & 0 & -27 & 0 & 174 & 300 & 174 & 0 & -27 & 0 & 3\\
    6 & 6 & -50 & 0 & 300 & 512 & 300 & 0 & -50 & 6 & 6\\
    3 & 0 & -27 & 0 & 174 & 300 & 174 & 0 & -27 & 0 & 3\\
    0 & 0 &  0  & 0 &  0  &  0  &  0  & 0 &  0  & 0 & 0\\
    0 & 0 &  2  & 0 & -27 & -50 & -27 & 0 &  2  & 0 & 0\\
    0 & 0 &  0  & 0 &  0  &  0  &  0  & 0 &  0  & 0 & 0\\
    0 & 0 &  0  & 0 &  3  &  6  &  3  & 0 &  0  & 0 & 0
    \end{bmatrix}$}}\hfill
  \parbox{8cm}{\centering\includegraphics{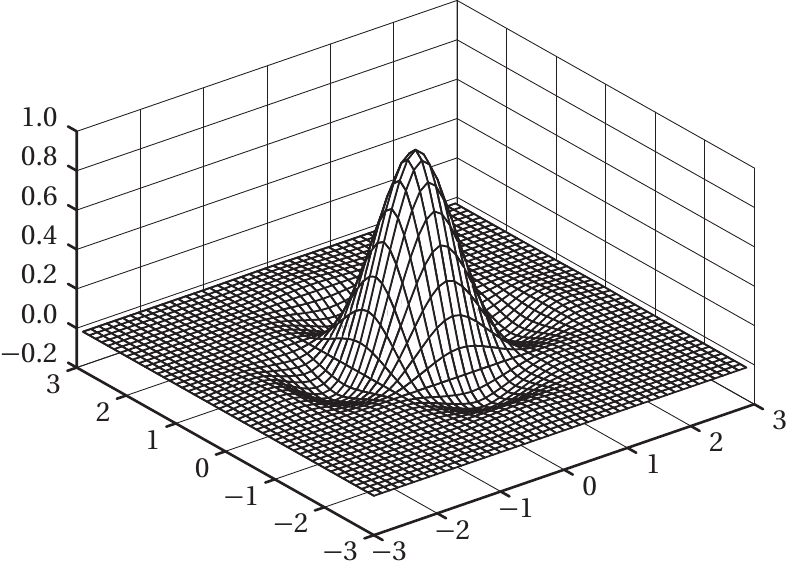}}\par
  \caption{Masks and graphs of the basic limit functions for the pseudo-splines with $n=3$.}
  \label{fig:example-n=3}
\end{figure}

The examples in Figures~\ref{fig:example-n=1}--\ref{fig:example-n=3}, as well as many numerical experiments performed for $0 \le l < n \le 20$, suggest that the subdivision schemes with symbols in~\eqref{eq:family} are convergent. However, a systematic and theoretical analysis of convergence, though very crucial, is beyond the scope of this paper, and so are further investigations regarding other properties such as stability and degree of smoothness of the basic limit functions.


\section{Conclusion}\label{sec:conclusion}
To conclude this paper, we want to point out that the family of bivariate four-directional pseudo-splines in~\eqref{eq:family} is not unique. For example, if $n-l$ is odd, then the scheme with symbol
\[
  \check{a}^l_n(\bz)
  = a^l_n(\bz) + a^0_{n-l-1}(\bz) \sum_{j=1}^l \mu_j {\bpi(\bz)}^{(l+1-j,j)}
\]
for any set of weights $\mu_1,\dots,\mu_l\in\RR$ with $\mu_j=\mu_{l+1-j}$, $j=1,\dots,l$ is symmetric and has the same generation and reproduction degree and the same support as $a^l_n$. In fact, since $a^0_{n-l-1}\in\cI_{2n-2l-2}$ and ${\bpi(\bz)}^{(l+1-j,j)}\in\cJ_{2l+2}\subset\cI_{2l+2}$ for $j=1,\dots,l$, the difference $\check{a}^l_n-a^l_n$ is both in $\cI_{2n}$ and in $\cJ_{2l+2}$, which explains the degrees of generation and reproduction, and the statement about the support follows as in the proof of Proposition~\ref{proposition:support}. However, in the special case of $l=n-1$ the symbol $\check{a}_n^l$ is interpolatory only if all weights $\mu_j$ are zero, as in Remark~\ref{remark:non-unique}. Our numerical investigations further indicate that for $n-l$ even, the members of our family with symbols $a_n^l$ are the unique minimally supported schemes with generation degree $2n-1$ and reproduction degree $2l+1$.



\begin{thebibliography}{10}

\bibitem{Cavaretta:1991:SS}
A.~S. Cavaretta, W.~Dahmen, and C.~A. Micchelli.
\newblock Stationary subdivision.
\newblock {\em Memoirs of the American Mathematical Society}, 93(453):vi+186,
  Sept. 1991.

\bibitem{Charina:2013:PRO}
M.~Charina and C.~Conti.
\newblock Polynomial reproduction of multivariate scalar subdivision schemes.
\newblock {\em Journal of Computational and Applied Mathematics}, 240:51--61,
  Mar. 2013.

\bibitem{Charina:2011:SMS}
M.~Charina, C.~Conti, K.~Jetter, and G.~Zimmermann.
\newblock Scalar multivariate subdivision schemes and box splines.
\newblock {\em Computer Aided Geometric Design}, 28(5):285--306, June 2011.

\bibitem{Conti:2013:ACA}
C.~Conti, L.~Gemignani, and L.~Romani.
\newblock A constructive algebraic strategy for interpolatory subdivision
  schemes induced by bivariate box splines.
\newblock {\em Advances in Computational Mathematics}, 39(2):395--424, Aug.
  2013.

\bibitem{Conti:2011:PRF}
C.~Conti and K.~Hormann.
\newblock Polynomial reproduction for univariate subdivision schemes of any
  arity.
\newblock {\em Journal of Approximation Theory}, 163(4):413--437, Apr. 2011.

\bibitem{Deslauriers:1989:SII}
G.~Deslauriers and S.~Dubuc.
\newblock Symmetric iterative interpolation processes.
\newblock {\em Constructive Approximation}, 5(1):49--68, Dec. 1989.

\bibitem{Dong:2007:PWA}
B.~Dong and Z.~Shen.
\newblock Pseudo-splines, wavelets and framelets.
\newblock {\em Applied and Computational Harmonic Analysis}, 22(1):78--104,
  Jan. 2007.

\bibitem{Dyn:2008:PRB}
N.~Dyn, K.~Hormann, M.~A. Sabin, and Z.~Shen.
\newblock Polynomial reproduction by symmetric subdivision schemes.
\newblock {\em Journal of Approximation Theory}, 155(1):28--42, Nov. 2008.

\bibitem{Dyn:2002:SSI}
N.~Dyn and D.~Levin.
\newblock Subdivision schemes in geometric modelling.
\newblock {\em Acta Numerica}, 11:73--144, Jan. 2002.

\bibitem{Han:1998:OIS}
B.~Han and R.-Q. Jia.
\newblock Optimal interpolatory subdivision schemes in multidimensional spaces.
\newblock {\em SIAM Journal on Numerical Analysis}, 36(1):105--124, Feb. 1998.

\bibitem{Jia:2002:APO}
R.-Q. Jia and Q.~Jiang.
\newblock Approximation power of refinable vectors of functions.
\newblock In D.~Deng, D.~Huang, R.-Q. Jia, W.~Lin, and J.~Wang, editors, {\em
  Wavelet Analysis and Applications}, volume~25 of {\em AMS/IP Studies in
  Advanced Mathematics}, pages 153--176. American Mathematical
  Society/International Press, 2002.

\bibitem{Jia:2003:SAO}
R.-Q. Jia and Q.~Jiang.
\newblock Spectral analysis of the transition operators and its applications to
  smoothness analysis of wavelets.
\newblock {\em SIAM Journal on Matrix Analysis and Applications},
  24(4):1071--1109, 2003.

\bibitem{Levin:2003:PGA}
A.~Levin.
\newblock Polynomial generation and quasi-interpolation in stationary
  non-uniform subdivision.
\newblock {\em Computer Aided Geometric Design}, 20(1):41--60, Mar. 2003.

\bibitem{Sauer:2002:PII}
T.~Sauer.
\newblock Polynomial interpolation, ideals and approximation order of
  multivariate refinable functions.
\newblock {\em Proceedings of the American Mathematical Society},
  130(11):3335--3347, Nov. 2002.

\end{thebibliography}

\end{document}